\documentclass[12pt,oneside]{amsart}
\usepackage[utf8]{inputenc}
\usepackage{amsmath}
\usepackage{tikz-cd}
\usepackage{bm}
\usepackage{amssymb}
\usepackage{mathtools}
\usepackage{amsthm}
\usepackage{graphicx}
\usepackage{comment}
\usepackage{comment}
\usepackage{appendix}
\usepackage{caption}
\usepackage{subcaption}
\usepackage{hyperref}
\usepackage{upgreek}
\usepackage{enumitem}
\usepackage{comment}
\usepackage{tikz}
\usepackage{pict2e,mathrsfs}

\title[Reflecting Perfection for Finite Dimensional DGAs]{Reflecting Perfection for Finite Dimensional Differential Graded Algebras}
\author{Isambard Goodbody}
\date{\today}

\DeclareMathOperator{\rad}{rad}

\DeclareMathOperator{\Hom}{Hom}
\DeclareMathOperator{\Tor}{Tor}
\DeclareMathOperator{\RHom}{RHom}
\DeclareMathOperator{\End}{End}

\DeclareMathOperator{\thick}{thick}

\DeclareMathOperator{\Mod}{Mod}
\DeclareMathOperator{\proj}{proj}
\DeclareMathOperator{\fdmod}{mod}
\DeclareMathOperator{\perf}{perf}
\DeclareMathOperator{\Aus}{Aus}
\DeclareMathOperator{\Ext}{Ext}
\DeclareMathOperator{\lhf}{lhf}
\DeclareMathOperator{\rhf}{rhf}

\DeclareMathOperator{\f}{sf}
\DeclareMathOperator{\cf}{cf}

\begin{document}

\usetikzlibrary{matrix}
\usetikzlibrary{shapes}

\theoremstyle{plain}
\newtheorem{prop}{Proposition}[section]
\newtheorem{lemma}[prop]{Lemma}
\newtheorem{theorem}[prop]{Theorem}
\newtheorem{cor}[prop]{Corollary}
\newtheorem{conj}[prop]{Conjecture}

\theoremstyle{definition}
\newtheorem{defn}[prop]{Definition}
\newtheorem{ass}[prop]{Assumption}
\newtheorem{cons}[prop]{Construction}
\newtheorem{ex}[prop]{Example}
\newtheorem{remark}[prop]{Remark}
\newtheorem*{ack}{Acknowledgements}
\maketitle

\tikzset{
    vert/.style={anchor=south, rotate=90, inner sep=.5mm}
}

\newcommand{\rightarrowdbl}{\rightarrow\mathrel{\mkern-14mu}\rightarrow}

\newcommand{\xrightarrowdbl}[2][]{%
  \xrightarrow[#1]{#2}\mathrel{\mkern-14mu}\rightarrow
}

\begin{abstract}
We generalise two facts about finite dimensional algebras to finite dimensional differential graded algebras. The first is the Nakayama Lemma and the second is that the simples can detect finite projective dimension. We prove two dual versions which relate to Gorenstein differential graded algebras and Koszul duality respectively. As an application, we prove a corepresentability result for finite dimensional differential graded algebras.
\end{abstract}

\tableofcontents

\section{Introduction}

If $\mathcal{T}$ is a triangulated category generated by an object $g$, then the graded endomorphism ring $\End^\ast_\mathcal{T}(g)$ of $g$ does not in general determine $\mathcal{T}$. However if $\mathcal{T}$ has a suitable enhancement, then $\End^\ast_{\mathcal{T}}(g)$ admits a higher structure which determines $\mathcal{T}$ entirely. In the case of differential graded enhancements, the higher algebraic gadget is a differential graded algebra (DGA) whose cohomology ring is $\End^\ast_{\mathcal{T}}(g)$. Alongside this there is another discrete ring one can associate to a DGA: simply forget the differential and grading. We will use this discrete ring to study the higher structure of DGAs.
\\

A motivating example is provided by algebraic geometry where $\mathcal{T}$ is the derived category of a quasi-compact quasi-separated scheme. By Corollary 3.1.8 in \cite{VB02}, $\mathcal{T}$ is equivalent to the derived category $\mathcal{D}(A)$ of a DGA $A$. For this reason, the derived category of a DGA can be  thought of as the derived category of a scheme (which may or may not exist). 
\\

To combat their wild behaviour, we place some assumptions on our DGAs. The strongest assumption would be that the DGA is concentrated in degree zero. In this case there is no higher structure, and we are in the world of representation theory. A more generous assumption is that of connectivity (cohomology vanishes in positive degrees). In this case $\mathcal{D}(A)$ admits the standard $t$-structure which, for example, can be used to compute invariants via the heart. The noncommutative version of a proper scheme is a DGA with finite dimensional total cohomology. A finite dimensional DGA (fd DGA) is a DGA whose underlying chain complex is totally finite dimensional. In general, this is a strictly stronger assumption than properness. However, in \cite{RS20} it is shown that every proper connective DGA admits a finite dimensional model. 
\\

In Definition 2.3 of \cite{Orl20}, Orlov constructs the radical $J_-$ of a fd DGA from the radical $J$ of its underlying algebra. The DG module $A/J_-$ behaves like the semisimple quotient of a finite dimensional algebra and it can be used to construct a derived version of the radical filtration. We will use it to prove analogues of two facts in the theory of finite dimensional algebras. The first is a derived version of the Nakayama Lemma.

\begin{theorem}
Suppose $A$ is a fd DGA such that $A/J$ is $k$-separable. If $M \in \mathcal{D}(A)$ is such that $M \otimes_A^{\mathbb{L}} A/J_- \simeq 0$ then $M \simeq 0$. 
\end{theorem}

The second result generalises the fact that the simples of a finite dimensional algebra can detect if a module has finite projective dimension. For a DGA $A$ the corresponding notion is of a perfect $A$-module. These form the subcategory $\mathcal{D}^{\perf}(A) \subseteq \mathcal{D}(A)$.

\begin{theorem} \label{reflperfintro}
Suppose $A$ is a fd DGA such that $A/J$ is $k$-separable. If $M \in \mathcal{D}(A)$ is such that $M \otimes_A^{\mathbb{L}} A/J_-$ has finite dimensional total cohomology then $M \in \mathcal{D}^{\perf}(A)$. 
\end{theorem}

In other words, perfection is reflected along the functor 
\[
-\otimes^{\mathbb{L}}_A A/J_- \colon \mathcal{D}(A) \longrightarrow  \mathcal{D}(A/J_-)
\]

By replacing $- \otimes_A^\mathbb{L} A/J_-$ with each of the functors $\RHom_A(A/J_-,-)$ and $\RHom_A(-,A/J_-)$ we derive two dual versions of Theorem \ref{reflperfintro}. The former (Corollary \ref{Gorcondit}) gives a new characterisation of Gorenstein DGAs. The latter (Theorem \ref{Kosequiv}) relates $A$ to its Koszul dual 
\[
A^! := \RHom_A(A/J_-,A/J_-).
\]
We let $\mathcal{D}_{\cf}^{\perf}(A^!) \subseteq \mathcal{D}(A^!)$ consist of perfect modules with finite dimensional total cohomology. Let $\mathcal{D}_{\f}(A)$ be the smallest thick subcategory of $\mathcal{D}(A)$ generated by finite dimensional $A$-modules. 

\begin{theorem}
Let $A$ be a fd DGA such that $A/J$ is $k$-separable. Then there are equivalences
\[
\mathcal{D}_{\f}(A) \simeq \mathcal{D}^{\perf}(A^!) \quad \text{ and } \quad \mathcal{D}^{\perf}(A) \simeq \mathcal{D}_{\cf}^{\perf}(A^!)
\]
\end{theorem}

As an application we consider a corepresentability problem.

\begin{theorem} \label{corepresentintro}
Let $A$ be fd DGA such that $A/J$ is $k$-separable. Every DG functor $F: \mathcal{D}_{\f}(A) \to \mathcal{D}^b(k)$ is quasi-isomorphic to $\RHom_A(M,-)$ for some $M \in \mathcal{D}^{\perf}(A)$ if and only if $\mathcal{D}_{\cf}(A^!) = \mathcal{D}^{\perf}_{\cf}(A^!)$.
\end{theorem}

\begin{ack}
I'd like to thank my supervisor Greg Stevenson as well as Dmitri Orlov for his explanation of Example \ref{sfneqcf}. I am supported by a PhD scholarship from the Carnegie Trust for the Universities of Scotland.
\end{ack}

\section{Preliminaries on Finite Dimensional DGAs} \label{prelims}

Let $k$ be a field and $\mathcal{C}(k)$ the category of chain complexes over $k$. A DGA $A$ is a monoid in $\mathcal{C}(k)$. The category of right $A$-modules will be denoted $\mathcal{C}(A)$. The derived category of $A$ is $\mathcal{D}(A)$ and $\mathcal{D}^{\perf}(A) \subseteq \mathcal{D}(A)$ consists of the perfect $A$-modules. An $A$-module is cohomologically finite if $H^i(M)$ is finite dimensional for all $i$ and $H^i(M) = 0$ for $\lvert i \rvert >> 0$. We let $\mathcal{D}_{\cf}(A) \subseteq \mathcal{D}(A)$ consist of the cohomologically finite modules. So for a finite dimensional algebra $\Lambda$, $\mathcal{D}^{\perf}(\Lambda) \simeq K^b(\proj \Lambda)$ and $\mathcal{D}_{\cf}(\Lambda) \simeq \mathcal{D}^b(\fdmod \Lambda)$. We will write $\mathcal{D}^b(k)$ for $\mathcal{D}^{\perf}(k) = \mathcal{D}_{\cf}(k)$.

The chain complexes of morphisms in the DG categories associated to $\mathcal{C}(A)$ and $\mathcal{D}(A)$ will be denoted $\Hom_A(M,N) \in \mathcal{C}(k)$ and $\RHom_A(M,N)$ $\in$ $\mathcal{C}(k)$ respectively. We let $A^{op}$ be the opposite DGA to $A$ and $A^e := A^{op} \otimes_k A$ the enveloping DGA of $A$.

If $\mathcal{T}$ is a triangulated category and $t \in \mathcal{T}$, $\thick_0(t)$ denotes all finite sums of shifts of summands of $t$. And for $n \in \mathbb{N}$ $\thick_{n+1}(t)$  consists of all summands of objects $s$ for which are there triangles 
\[
s' \longrightarrow s \longrightarrow s'' \longrightarrow
\]
where $s' \in \thick_{n}(t)$ and $s'' \in \thick_0(t)$. Then $\thick(t) = \bigcup_{n \geq 0} \thick_n(t)$ is the smallest thick subcategory containing $t$.

\begin{defn} For a DGA $A$
\begin{enumerate} [label = (\roman*)]
\item $A$ is a finite dimensional DGA (fd DGA) if $A^i$ is finite dimensional for all $i$ and $A^i = 0$ for $\lvert i \rvert > > 0$. 
\item $A$ is proper if $H^i(A)$ is finite dimensional for all $i$ and $H^i(A) = 0$ for $\lvert i \rvert > > 0$.
\item $A$ is smooth if $A \in \mathcal{D}^{\perf}(A^e)$.
\item $A$ is regular if $\mathcal{D}^{\perf}(A) = \thick_N(A)$ for some $N \in \mathbb{N}$.
\end{enumerate}
\end{defn}

By Corollary 2.6 in \cite{Orl20} the class of DGAs which are quasi-isomorphic to fd DGAs is closed under derived equivalence. There is a complete characterisation of this class due to Orlov. Throughout this paper we will use the notation $J = J(A)$ for the Jacobson radical of the underlying algebra of a fd DGA $A$.

\begin{theorem} [Corollary 2.20 \cite{Orl20}] \label{orlovcatres} A proper DGA $A$ is quasi-isomorphic to a fd DGA if and only if there is a proper DGA $E$ and a fully faithful functor
\[
\mathcal{D}^{\perf}(A) \hookrightarrow \mathcal{D}^{\perf}(E)
\]
such that $\mathcal{D}^{\perf}(E)$ has a full exceptional collection. If $A/J$ is $k$-separable then $E$ is smooth and then $A$ is smooth if and only if the embedding is admissible. 
\end{theorem}

We note a consequence for Hochschild homology of fd DGAs. Recall the Hochschild homology of a DGA $A$ is $HH_\ast(A) := H^\ast(A \otimes_{A^e}^{\mathbb{L}} A)$.

\begin{cor} \label{Hochshom}
Suppose $A$ is a smooth fd DGA such that $A/J$ is $k$-separable. Then the Hochschild homology $HH_\ast(A)$ of $A$ is concentrated in degree zero. 
\end{cor}

\begin{proof}
In this case $\mathcal{D}^{\perf}(A)$ is admissible in $\mathcal{D}^{\perf}(E)$ where $E$ is as in Theorem \ref{orlovcatres}. Hochschild homology is an additive invariant and so $HH_\ast(A)$ is a summand of $HH_\ast(E)$ (see for example \cite{Kuz09}). Since $E$ admits a full exceptional collection $HH_\ast(E) \cong HH_\ast(S)$ where $S$ is a separable $k$-algebra. In the notation of \cite{Orl20}, $S = \prod_{i=1}^n \End_{\mathcal{E}}(K_i)$. The separability follows from Theorem 2.18 5) in loc.\ cit.\ since $A/J$ is $k$-separable. Since $S$ is $k$-separable so is $S^e$ and in particular it has global dimension 0. Hence $HH_\ast(S)$ is concentrated in degree zero.  
\end{proof}

\begin{comment}
If $A$ is a smooth finite dimensional algebra then $A/J$ is automatically $k$-separable by an argument of Rickard. If A is a smooth fd DGA in non-positive degrees, then 
\[
A \to H^0(A) \to H^0(A)/\rad H^0(A) = A/J
\]
and so since A/J is perfect over A,  H^0(A)/\rad H^0(A) is perfect over H^0. Hence H^0(A) is finite global dimension. By Rickard's argument A/J is k-separable. The HKR theorem applied to a scheme $X$ requires the characteristic of $k$ to be greater than $\dim(X)$. Therefore positive genus smooth projective curves are not derived equivalent to a proper DGA over any field.
\end{comment}

\begin{remark} Using the Hochschild-Konstant-Rosenberg Theorem of \cite{Ma01} and \cite{Ma08} (and see also \cite{Kuz09} Theorem 8.3) one can produce many examples of smooth and proper schemes whose Hochschild homology is not concentrated in degree zero. Examples include smooth projective curves of genus $g > 0$, $K_3$ surfaces, many Fano 3-folds and so on. Hence none of these geometric objects can be derived equivalent to a fd DGA. This provides plenty of proper DGAs which do not admit finite dimensional models.
\end{remark}

Corollary \ref{Hochshom} was already proved for proper connective DGAs in \cite{shk07}. We note that smoothness simplifies for fd DGAs.

\begin{cor} Suppose $A$ is a fd DGA such that $A/J$ is $k$-separable. Then $A$ is smooth if and only if $\mathcal{D}^{\perf}(A) = \mathcal{D}_{\cf}(A)$. 
\end{cor}
\begin{comment}
\begin{proof}
It is known that for any smooth and proper DGA we have $\mathcal{D}^{\perf}(A) = \mathcal{D}_{\cf}(A)$. A proof can be found in \cite{KS22}. Conversely suppose $\mathcal{D}^{\perf}(A) = \mathcal{D}_{\cf}(A)$. For any fd DGA, there is an embedding 
\[
\mathcal{D}(A) \hookrightarrow \mathcal{D}(E)
\]
which restricts to $\mathcal{D}^{\perf}(A) \to \mathcal{D}^{\perf}(E)$ and whose right adjoint restricts to $\mathcal{D}_{\cf}(E) \to \mathcal{D}_{\cf}(A)$. So if $\mathcal{D}_{\cf}(A) = \mathcal{D}^{\perf}(A)$, then $\mathcal{D}^{\perf}(A)$ is right admissible in $\mathcal{D}^{\perf}(E)$. This means that $\langle \mathcal{D}^{\perf}(A)^\perp, \mathcal{D}^{\perf}(A) \rangle$ is a semi-orthogonal decomposition of $\mathcal{D}^{\perf}(E)$. Since $E$ is smooth, \cite{Lun09} implies that $\mathcal{D}^{\perf}(A)$ is smooth. Hence $A$ is smooth.
\end{proof}
\end{comment}
\begin{proof}
This is a special case of Proposition 1.11 in \cite{Orl23}. 
\end{proof}

\begin{ex} \label{sfneqcf}
As mentioned in Section 3.3 of \cite{Orl23}, there are modules over fd DGAs which have finite dimensional total cohomology but are not quasi-isomorphic to a module whose underlying chain complex is finite dimensional. An example is given by $k[x,y]/(x^6,y^3)$ with $\lvert x \rvert = 0, \lvert y \rvert =1$ and zero differential. Theorem 5.4 of \cite{Efi18} gives an example of a $k[x]/x^6$-$k[y]/y^3$ bimodule $V$ with finite dimensional cohomology. The theorem in loc.\ cit.\ implies that the associated gluing is a DGA which does not satisfy Theorem \ref{orlovcatres}. It follows that the gluing, and so also $V$, doesn't admit a finite dimensional model. 
\end{ex}

\begin{defn} A module $M$ over a fd DGA $A$ is strictly finite dimensional if $M^i$ is finite dimensional for all $i$ and $M^i = 0$ for $\lvert i \rvert >> 0$. Let $\mathcal{D}_{\f}(A)$ be the smallest thick subcategory of $\mathcal{D}_{\cf}(A)$ containing all strictly finite dimensional modules.
\end{defn}

The following is Proposition 3.9 in \cite{Orl23}.

\begin{prop} \label{sf=cf}
If $A$ is a fd DGA concentrated in non-positive degrees, then $\mathcal{D}_{\f}(A) = \mathcal{D}_{\cf}(A)$.
\end{prop}

One can check smoothness using this category too. 

\begin{prop} Let $A$ be a fd DGA such that $A/J$ is $k$-separable. Then $A$ is smooth if and only if $\mathcal{D}^{\perf}(A) = \mathcal{D}_{\f}(A)$. 
\end{prop}

\begin{proof}
This follows from Corollary 3.12 in \cite{Orl23}.
\end{proof}

The radical $J$ of a fd DGA is not in general a DG-ideal. The following construction of Orlov will be used many times. 

\begin{defn} \label{jacrad}
If $A$ is a fd DGA, let $J = J(A)$ denote the Jacobson radical of the underlying ungraded algebra of $A$. Then the external and internal DG ideals associated to $J$ are
\[
J_+ \coloneqq J+d(J) \quad \text { and } \quad J_- \coloneqq \{ r \in J \mid d(r) \in J\}.
\] 
\end{defn}

\begin{remark}
Since $J_- \subseteq J \subseteq J_+$ we deduce that $J_-$ is a nilpotent DG ideal and that the underlying algebra of the quotient $A/J_+$ is semisimple. Furthermore the inclusion $J_- \hookrightarrow J_+$ is a quasi-isomorphism.
\end{remark}

\begin{theorem} [Proposition 2.16 \cite{Orl20}] \label{semisimple}
Let $A$ be a fd DGA. Then 
\[
\mathcal{D}(A/J_+) \simeq \mathcal{D}(D_1 \times \cdots \times D_n)
\]
where $D_1,\dots,D_n$ are finite dimensional division algebras. 
\end{theorem}

The $D_i$ are exactly those appearing in the Artin-Wedderburn Theorem applied to the underlying algebra of $A/J_+$. 

\begin{defn}[Definition 2.11 \cite{Orl20}] A fd DGA $A$ is semisimple if it is derived equivalent to a finite product of finite dimensional division algebras.
\end{defn}

\begin{theorem}(The Radical Filtration) \label{radfilt}
Suppose $A$ is a fd DGA then, 
\[
\mathcal{D}_{\f}(A) = \thick(A/J_-).
\]
\end{theorem}

\begin{proof}
Let $M$ be a strictly finite dimensional $A$-module. Then there is a filtration of $M$ by strictly finite dimensional $A$-modules
\[
0 = MJ_-^N \subseteq MJ_-^{N-1} \subseteq \dots \subseteq MJ_- \subseteq M
\]
and the factors are $MJ_-^i/MJ_-^{i+1} \in \mathcal{D}(A/J_-) \simeq \mathcal{D}(A/J_+)$. So by Theorem \ref{semisimple},  $MJ_-^i/MJ_-^{i+1} \in \thick_0(A/J_-) \subseteq \mathcal{D}(A)$. The filtration then produces triangles in $\mathcal{D}(A)$ which exhibit $M \in \thick_N(A/J_-)$. Since $\mathcal{D}_{\f}(A)$ is the smallest thick subcategory containing all strictly finite dimensional modules it follows that $\mathcal{D}_{\f}(A) \subseteq \thick(A/J_-)$. The converse follows since $A/J_-$ is strictly finite dimensional. 
\end{proof}

\section{A Derived Nakayama Lemma}

To prove a derived version of the Nakayama Lemma for fd DGAs, we consider a bimodule version of the radical filtration. 

\begin{lemma} \label{envelopesemi}
Let $A$ be a fd DGA such that $A/J$ is $k$-separable. then $(A/J_-)^e$ and $(A/J_+)^e$ are semisimple DGAs. 
\end{lemma}

\begin{proof}
Since $A/J$ is $k$-separable, its quotient $A/J_+$ is too. By Lemma \ref{appendixlemma}, the underlying algebra of the enveloping DGA $(A/J_+)^e$ is semisimple. Hence by Theorem \ref{semisimple}, $(A/J_+)^e$ is a semisimple DGA. And so is $(A/J_-)^e \simeq (A/J_+)^e$.
\end{proof}

\begin{lemma} \label{thickA/J} Let $A$ be a fd DGA such that $A/J$ is $k$-separable. Let $N \in \mathbb{N}$ be such that $J^N = 0$ and $J^{N-1} \neq 0$. If $M \in \mathcal{D}(A)$, then $ M \in \thick_N( M \otimes_A^{\mathbb{L}} (A/J_-)^e) \subseteq \mathcal{D}(A)$. 
\end{lemma}

\begin{proof}
Note that $J_-$ is a two-sided ideal and the filtration 
\[
0 = (J_-)^n \subseteq (J_-)^{n-1} \subseteq \dots \subseteq J_- \subseteq A
\]
is of $A$ bimodules. Each of the factors $(J_-)^{i}/(J_-)^{i+1}$ are strictly finite dimensional $A/J_-$ bimodules. Then by Lemma \ref{envelopesemi}, $(J_-)^{i}/(J_-)^{i+1} \in \thick_0((A/J_-)^e)$. This exhibits $A$ as an object of $\thick_N((A/J_-)^e) \subseteq \mathcal{D}((A/J_-)^e)$. Then for any $M \in \mathcal{D}(A)$
\[
M \simeq M \otimes_A^{\mathbb{L}} A \in M \otimes_A^{\mathbb{L}} \thick_N ((A/J_-)^e ) \subseteq \thick_N( M \otimes_A^{\mathbb{L}} (A/J_-)^e  ). 
\]

\end{proof}

\begin{theorem} (Derived Nakayama Lemma) 
Let $A$ be a fd DGA such that $A/J$ is $k$-separable and let $M \in \mathcal{D}(A)$. If $M \otimes_A^{\mathbb{L}} A/J_- \simeq 0$, then $M \simeq 0$.
\end{theorem}

\begin{proof}
Suppose $M \otimes_A^{\mathbb{L}} A/J_- \simeq 0$. Consider the following thick subcategory which contains $A/J_-$. 
\[
\mathcal{C} \coloneq \{ X \in \mathcal{D}(A^{op}) \mid M \otimes_A^\mathbb{L} X = 0 \}
\]
Then $\mathcal{C}$ contains $\thick_{A^{op}}(A/J_-)$ which equals $\mathcal{D}_{\f}(A^{op})$ by Theorem \ref{radfilt}. Clearly  $(A/J_-)^e \in  \mathcal{D}_{\f}(A^{op}) \subseteq \mathcal{C}$. Hence $M \otimes_A^{\mathbb{L}} (A/J_-)^e \simeq 0$. Then the result follows from Lemma \ref{thickA/J}.
\end{proof}

\begin{theorem} \label{reflbound}Let $A$ be a fd DGA such that $A/J$ is $k$-separable. Suppose $M \in \mathcal{D}(A)$ is such that $M \otimes_A^{\mathbb{L}} A/J_- \in \mathcal{D}_{\cf}(A)$ then $M \in \mathcal{D}_{\cf}(A)$. 
\end{theorem}

\begin{proof} Let $M \in \mathcal{D}(A)$ and set
\[
\mathcal{C} \coloneq \{ X \in \mathcal{D}(A^{op}) \mid \dim H^\ast(M \otimes_A^\mathbb{L} X) < \infty \}.
\]
This is a thick subcategory as it is the preimage of the thick subcategory $\mathcal{D}^b(k)$ under the functor $M \otimes_A^{\mathbb{L}} -: \mathcal{D}(A^{op}) \to \mathcal{D}(k)$. It contains $A/J_-$ by assumption and so it contains $\mathcal{D}_{\f}(A^{op}) = \thick(A/J_-)$. Since $(A/J_-)^e$ is a strictly finite dimensional $A^{op}$ module, $(A/J_-)^e \in \mathcal{D}_{\f}(A^{op}) \subseteq \mathcal{C}$. Therefore by Lemma \ref{thickA/J}, 
\[
M \in \thick( M \otimes_A^\mathbb{L} (A/J_-)^e) \subseteq \mathcal{D}_{\cf}(A). \qedhere
\] 
\end{proof}

\begin{remark}
The same proof works if we replace the property of having finite dimensional cohomology with having finite dimensional cohomology in each degree or similar such variants. 
\end{remark}

\section{Reflecting Perfection}

The simples of a finite dimensional algebra can detect finite projective dimension in the following sense.

\begin{prop}
If $A$ is a finite dimensional algebra and $M \in \Mod A$ is such that $M \otimes^{\mathbb{L}} A/J \in \mathcal{D}^b(k)$ then $M$ has finite projective dimension.
\end{prop}

One can prove this using minimal projective resolutions and the ordinary Nakayama Lemma. We will generalise this result to fd DGAs using the following construction from Section 2.3 of \cite{Orl20}.

\begin{defn} \label{Auscons}
For a fd DGA $A$, let $J_i := (J^i)_-$ for $i =1,\dots,N$ where $J^N = 0$. Set
\[
M:= A/J_1  \oplus A/J_2 \oplus \dots \oplus A/J_{N-1} \oplus A. 
\]
The Auslander algebra $\Aus(A)$ of $A$ is  defined as  $\End_A(M)$. Define the $\Aus(A)$-modules $P_i := \Hom_{A}(M,A/J_i)$ for $i = 1,\dots,N$.
\end{defn}

This provides a non-commutative resolution of the fd DGA in the following sense.

\begin{theorem} [Theorem 2.19 \cite{Orl20}]
If $A$ is a fd DGA then 
\[
- \otimes_A^{\mathbb{L}} P_N: \mathcal{D}(A) \hookrightarrow \mathcal{D}(\Aus(A))
\]
is fully faithful and if $A/J$ is $k$-separable then $\Aus(A)$ is smooth.
\end{theorem}

The following was observed in Proposition 6.9 of \cite{KS22}.

\begin{lemma} \label{thomason}
For $M \in \mathcal{D}(A)$, if $M \otimes_A^{\mathbb{L}} P_N \in \mathcal{D}^{\perf}(\Aus(A))$, then $M \in \mathcal{D}^{\perf}(A)$.
\end{lemma}

\begin{proof}
Let $t_i \in \mathcal{D}(A)$ be a family of objects, then since $F \coloneqq - \otimes_A^{\mathbb{L}} P_N $ is fully faithful and preserves coproducts
\begin{align*}
\RHom_A\left(M , \bigoplus t_i\right) & \simeq \RHom_{\Aus(A)}\left(FM, F\bigoplus t_i\right) \\
& \simeq \bigoplus \RHom_{\Aus(A)}\left(FM, Ft_i\right) \\
& \simeq \bigoplus \RHom_A\left(M,t_i\right)
\end{align*}
and the composition provides an inverse to the canonical map. 
\end{proof}

\begin{lemma} \label{keylemma}
There exists an $S \in \mathcal{D}_{\f}(A^{op})$ with the following property. If $M \in \mathcal{D}(A)$ with $M \otimes_A^{\mathbb{L}} S \in \mathcal{D}^b(k)$ then $M  \in \mathcal{D}^{\perf}(A)$. 
\end{lemma}

\begin{proof}
Let $E \coloneqq\Aus(A)$ and $J(E)_-$ the internal DG ideal associated to the radical of the fd DGA $E$. Consider the composition
\[
\mathcal{D}(A) \hookrightarrow \mathcal{D}(E) \rightarrow \mathcal{D}(E/J(E)_-) 
\]
of $- \otimes_A^\mathbb{L} P_N$ and $-\otimes_E^{\mathbb{L}} E/J(E)_-$. By Lemma \ref{thomason}, the first functor reflects perfection. Since $E$ is smooth, $\mathcal{D}_{\cf}(E) = \mathcal{D}^{\perf}(E)$ and so by Corollary \ref{reflbound} i), the second functor reflects perfection. Therefore, if $M \otimes^{\mathbb{L}}_A P_N \otimes^{\mathbb{L}}_E E/J(E)_-$ has finite dimensional cohomology then $M \in \mathcal{D}^{\perf}(A)$. It remains to show that $P_N \otimes^{\mathbb{L}}_E E/J(E)_- \in \mathcal{D}_{\f}(A^{op})$. This follows since $P_N$ is a summand of $E$ and so $P_N \otimes^{\mathbb{L}}_E E/J(E)_-$ can be modelled as $P_N \otimes_E E/J(E)_-$ which is strictly finite dimensional.
\end{proof}

\begin{theorem} [Reflecting Perfection] \label{reflperf}
Let $A$ be a fd DGA such that $A/J$ is $k$-separable. If $M  \otimes^{\mathbb{L}}_A A/J_- \in \mathcal{D}^b(k)$ then $M \in \mathcal{D}^{\perf}(A)$. 
\end{theorem}

\begin{proof}
By Theorem \ref{radfilt} applied to $A^{op}$, $\mathcal{D}_{\f}(A^{op}) = \thick(A/J_-)$. Now consider the subcategory 
\[
\mathcal{C} \coloneq \{ N \mid M \otimes_A^{\mathbb{L}} N \in \mathcal{D}^b(k) \} \subseteq{\mathcal{D}(A^{op})}.
\]
Then $\mathcal{C}$ is a thick subcategory and by assumption it contains $A/J_-$. Hence it contains all of $\mathcal{D}_{\f}(A^{op})$. In particular it contains $S$ as in Lemma \ref{keylemma}. Therefore $M \otimes_A^{\mathbb{L}} S$ has finite dimensional cohomology so $M \in \mathcal{D}^{\perf}(A)$.  
\end{proof}

Since $-\otimes_A^\mathbb{L} A/J_-$ also preserves perfection this gives another characterisation of the perfect modules.

\begin{cor} \label{descperf}
If $A$ is a fd DGA such that $A/J$ is $k$-separable, then 
\[
\mathcal{D}^{\perf}(A) = \{ M \in \mathcal{D}(A) \mid M \otimes_A^{\mathbb{L}} A/J_- \in \mathcal{D}^b(k) \}.
\]
\end{cor}

\section{The Codual version and Gorenstein DGAs} \label{GorandDual}

We consider a dual version of Corollary \ref{descperf} where $-\otimes_A^{\mathbb{L}} A/J_-$ is replaced with $\RHom_A(A/J_-,-)$. We first recall some facts about Ext-Tor duality. The $k$-dual functor, 
\[
(-)^\vee \coloneq \Hom_k(-,k) : \mathcal{D}(A)^{op} \rightarrow \mathcal{D}(A^{op})
\]
restricts to an equivalence $\mathcal{D}_{\cf}(A)^{op} \simeq \mathcal{D}_{\cf}(A^{op})$. If $A$ is a fd DGA then it also restricts to an equivalence $\mathcal{D}_{\f}(A) \simeq \mathcal{D}_{\f}(A^{op})^{op}$.

\begin{prop} [$\Ext$-$\Tor$ Duality]
Let $A$ be a DGA, then for any $M \in \mathcal{D}(A)$ and $N \in \mathcal{D}(A^{op})$
\[
\RHom_A(M,N^\vee) \simeq (M \otimes_A^{\mathbb{L}} N)^\vee 
\]
\end{prop}

\begin{proof}
We note that $(-)^\vee = \Hom_k(-,k) \simeq \RHom_k(-,k)$. So the result follows from the derived Hom-Tensor adjunction
\[
 \RHom_A(M,\RHom_k(N,k))  \simeq \RHom_k( M 
 \otimes_A^{\mathbb{L}} N,k) \qedhere
\]
\end{proof}

Suppose $A$ is a proper DGA so $\mathcal{D}^{\perf}(A) \subseteq \mathcal{D}_{\cf}(A)$. Then we can restrict $(-)^\vee$ to an equivalence
\begin{equation} \label{dualperfequiv}
(-)^\vee: \mathcal{D}^{\perf}(A) = \thick(A) \xrightarrow{\sim} \thick_{A^{op}}(A^\vee)^{op} \subseteq \mathcal{D}_{\cf}(A^{op})^{op} 
\end{equation}

This is related to the following definition in \cite{Jin20}.

\begin{defn}
A DGA $A$ is Gorenstein if $\mathcal{D}^{\perf}(A) = \thick_A(A^\vee)$.
\end{defn}

\begin{prop} \label{Gorensteinvskdual}Let $A$ be a proper DGA. The functor $(-)^\vee$ restricts to an equivalence $\mathcal{D}^{\perf}(A) \simeq \mathcal{D}^{\perf}(A^{op})^{op}$ if and only if $A$ is Gorenstein.
\end{prop}

\begin{proof}
By Equation (\ref{dualperfequiv}), the functor $(-)^\vee$ restricts to an equivalence on $\mathcal{D}^{\perf}(A)$ if and only if $\mathcal{D}^{\perf}(A^{op}) = \thick_{A^{op}}(A^\vee)$. This is exactly the condition that $A^{op}$ is Gorenstein. The result then follows since $A$ is Gorenstein if and only if $A^{op}$ is. Indeed the equivalence $(-)^\vee: \mathcal{D}_{\cf}(A) \simeq \mathcal{D}_{\cf}(A^{op})^{op}$ sends $\mathcal{D}^{\perf}(A)$ to $\thick_{A^{op}}(A^\vee)$ and $\thick_A(A^\vee)$ to $\mathcal{D}^{\perf}(A^{op})$. Therefore $\mathcal{D}^{\perf}(A) = \thick_A(A^\vee)$ if and only if $\thick_{A^{op}}(A^\vee) = \mathcal{D}^{\perf}(A^{op})$.
\end{proof}

There is significant room for confusion here: One always has an equivalence $\mathcal{D}^{\perf}(A) \simeq \mathcal{D}^{\perf}(A^{op})^{op}$ but in general it is not given by $(-)^\vee$.

\begin{prop} \label{otherduality}
For any DGA $A$, there is an equivalence
\[
\RHom_A(-,A): \mathcal{D}^{\perf}(A) \xrightarrow{\sim} \mathcal{D}^{\perf}(A^{op})^{op}.
\]
\end{prop}

\begin{proof}
The natural transformation
\[
\RHom_A(M,N) \longrightarrow \RHom_{A^{op}}(\RHom_A(N,A),\RHom_A(M,A))
\]
is an equivalence at $N = A$ and the set of objects $N$ such that this natural transformation is an equivalence form a thick subcategory. Therefore it holds for all $N \in \mathcal{D}^{\perf}(A)$. Hence the functor is fully faithful. Its image contains $A \simeq \RHom_A(A,A)$ and so it is essentially surjective.  
\end{proof}

One justification for this definition is the relationship to Serre functors in the sense of \cite{BoKa}. This generalises a well known result for finite dimensional algebras. 

\begin{prop}
A proper DGA is Gorenstein if and only if $\mathcal{D}^{\perf}(A)$ admits a Serre functor. In this case the Serre functor is
\[
- \otimes_A^\mathbb{L} A^\vee: \mathcal{D}^{\perf}(A) \xrightarrow{\sim} \mathcal{D}^{\perf}(A)
\]

\end{prop}

\begin{proof}
Suppose $\mathcal{D}^{\perf}(A)$ admits a Serre functor $S: \mathcal{D}^{\perf}(A) \xrightarrow{\sim} \mathcal{D}^{\perf}(A)$. The functor 
\[
\Hom_{\mathcal{D}(A)}(M,-)^\vee: \mathcal{D}(A) \to \fdmod(k)^{op}
\]
is representable by $S(M)$. By Lemma 2.2, in \cite{KS22} this lifts to an equivalence  $\RHom_A(M,-)^\vee \simeq \RHom_A(-,S(M)) \in \mathcal{D}(\mathcal{D}^{\perf}(A))$.
Hence 
\[
\RHom_A(-,S(A)) \simeq \RHom_A(A,-)^\vee \simeq (-)^\vee \simeq \RHom_A(-,A^\vee)
\]
where the last equivalence follows from Ext-Tor duality. So by Yoneda, $A^\vee \simeq S(A) \in \mathcal{D}^{\perf}(A)$. It is clear that if $\mathcal{D}^{\perf}(A)$ admits a Serre functor then so does $\mathcal{D}^{\perf}(A)^{op}$. So one also has $A^\vee \in \mathcal{D}^{\perf}(A^{op})$. Then $\thick_A(A^\vee) \subseteq \mathcal{D}^{\perf}(A)$ and  $\thick_{A^{op}}(A^\vee) \subseteq \mathcal{D}^{\perf}(A^{op})$. Applying $(-)^\vee$ to the latter gives that $\mathcal{D}^{\perf}(A) \subseteq \thick_A(A^\vee)$. Hence $A$ is Gorenstein.
\\

Conversely by Propositions \ref{Gorensteinvskdual} and \ref{otherduality}, $\RHom(-,A)^\vee \simeq - \otimes_A A^\vee$ is an equivalence. For $M,N \in \mathcal{D}^{\perf}(A)$ there are natural equivalences
\[
\RHom_A(M,\RHom_A(N,A)^\vee)^\vee \simeq M \otimes_A \RHom_A(N,A) \simeq \RHom_A(N,M)
\]
where the last equivalence holds since it holds at $N = A$ and the class of objects for which it is an equivalence is thick. 
\end{proof}

This description of the Serre functor appears in \cite{shk07}. The following is a dual version of Corollary \ref{descperf}. 

\begin{prop} \label{dualperfdesc}
For a fd DGA $A$ such that $A/J$ is $k$-separable,
\[
\thick_A(A^\vee) = \{ M \in \mathcal{D}_{\cf}(A) \mid \RHom(A/J_-,M) \in \mathcal{D}^b(k) \}.
\]
\end{prop}

\begin{proof}
Corollary \ref{descperf} applied to $A^{op}$ implies
\[
\mathcal{D}^{\perf}(A^{op}) = \{ M \in \mathcal{D}_{\cf}(A^{op}) \mid A/J_- \otimes_A^\mathbb{L} M \in \mathcal{D}^b(k) \} . 
\]
Applying the equivalence $(-)^\vee$ gives 
\[
\thick_A(A^\vee) = \{ M^\vee \in \mathcal{D}_{\cf}(A) \mid A/J_- \otimes_A^\mathbb{L} M \in \mathcal{D}^b(k) \}.
\]
This can be rewritten using the equivalence $(-)^\vee$ as
\[
\thick_A(A^\vee) = \{ N \in \mathcal{D}_{\cf}(A) \mid A/J_- \otimes_A^\mathbb{L} N^\vee \in \mathcal{D}^b(k) \}.
\]
Now, $A/J_- \otimes_A^\mathbb{L} N^\vee \in \mathcal{D}^b(k)$ if and only if $(A/J_- \otimes_A^\mathbb{L} N^\vee)^\vee \in \mathcal{D}^b(k)$. Hence the result follows since
\[
(A/J_- \otimes_A^\mathbb{L} N^\vee)^\vee \simeq \RHom_A(A/J_-, N^{\vee\vee}) \simeq  \RHom_A(A/J_-, N). \qedhere
\]
\end{proof}

We deduce that reflecting perfection for $\RHom(A/J_-,-)$ is related to being Gorenstein.

\begin{cor} \label{Gorcondit}
Let $A$ be a fd DGA such that $A/J$ is $k$-separable. Then $A$ is Gorenstein if and only if both of the functors
\begin{align*}
\RHom_A(A/J_-,-) :& \mathcal{D}_{\cf}(A) \to \mathcal{D}(A/J_-) \\
\RHom_{A^{op}}(A/J_-,-) :& \mathcal{D}_{\cf}(A^{op}) \to \mathcal{D}(A^{op}/J_-) 
\end{align*}
reflect perfection.
\end{cor}

\begin{proof}
Proposition \ref{dualperfdesc} applied to $A$ and $A^{op}$ implies that $\thick_A(A^\vee) \subseteq \mathcal{D}^{\perf}(A)$ and $\thick_{A^{op}}(A^\vee) \subseteq \mathcal{D}^{\perf}(A^{op})$. So applying $(-)^\vee$ to the latter gives that $\mathcal{D}^{\perf}(A) \subseteq \thick_A(A^\vee)$. Hence $A$ is Gorenstein. If $A$ is Gorenstein then so is $A^{op}$. Then the result follows from Proposition \ref{dualperfdesc} applied to $A$ and $A^{op}$ respectively.
\end{proof}

\begin{cor}
Let $A$ be a fd DGA such that $A/J$ is $k$-separable. Then both $\RHom_A(A/J_-,A)$ and $\RHom_{A^{op}}(A/J_-,A)$ have finite dimensional cohomology if and only if $A$ is Gorenstein. 
\end{cor}

\begin{proof}
Suppose $A$ is Gorenstein, then so is $A^{op}$. By Proposition \ref{dualperfdesc}, $\RHom_A(A,A/J_-)$ and $\RHom_{A^{op}}(A,A/J_-)$ are in $\mathcal{D}^b(k)$. Conversely, it follows from Proposition \ref{dualperfdesc} that $A \in \thick_{A}(A^\vee)$ and $A \in \thick_{A^{op}}(A^\vee)$. So $\mathcal{D}^{\perf}(A) \subseteq \thick_A(A^\vee)$ and $\mathcal{D}^{\perf}(A^\vee) \subseteq \thick_{A^{op}}(A^\vee)$. Applying $(-)^\vee$ to the latter gives that $\thick_A(A) \subseteq \mathcal{D}^{\perf}(A)$. 
\end{proof}

\section{The Contradual version and Koszul Duality}

We prove the functor $\RHom_A(-,A/J_-)$ reflects perfection and use it to study the Koszul dual. 

\begin{theorem} \label{descperfdual2}
Let $A$ be a fd DGA such that $A/J$ is $k$-separable then 
\[
\mathcal{D}^{\perf}(A) = \{ X \in \mathcal{D}(A) \mid \RHom_A(X,A/J_-) \in \mathcal{D}^b(k) \}.
\]
\end{theorem}

\begin{proof}
By Ext-Tor duality, 
\[
\RHom_A(-,A/J_-) \simeq \RHom_A(-,(A/J_-)^{\vee\vee}) \simeq (- \otimes_A^{\mathbb{L}} A/J_-^\vee)^\vee
\]
and so to prove the claim we must show that $- \otimes_A (A/J_-)^\vee$ preserves and reflects perfection. It preserves it since $(A/J_-)^\vee \in \mathcal{D}^b(k)$. For reflection note that $A/J_-^\vee$ generates $\mathcal{D}_{\f}(A^{op})$. Indeed  $A/J_-$ generates $\mathcal{D}_{\f}(A)$ and  $(-)^\vee: \mathcal{D}_{\f}(A) \simeq \mathcal{D}_{\f}(A^{op})^{op}$ is an equivalence. So the module $S$ of Lemma \ref{keylemma} is contained in $\thick_A(A/J_-^\vee) = \mathcal{D}_{\f}(A)$. As in the proof of Theorem \ref{reflperf}, $M \otimes_A^\mathbb{L} A/J_-^\vee \in \mathcal{D}^b(k)$ implies that $M \otimes_A^{\mathbb{L}} S \in \mathcal{D}^b(k)$. So Lemma \ref{keylemma} implies $M \in \mathcal{D}^{\perf}(A)$. 
\end{proof}

Consider the following definition from Section 4.1 of \cite{KS22}.

\begin{defn} \label{defnlhfrhf}
Let $\mathcal{T}$ be a $k$-linear triangulated category, then 
\begin{align*}
\mathcal{T}^{\lhf} &:= \{ M \in \mathcal{T} \mid \bigoplus_i \Hom_{\mathcal{T}}(M,\Sigma^i N) \text{ finite dimensional for all } N \in \mathcal{T} \} \\
\mathcal{T}^{\rhf} &:= \{ N \in \mathcal{T} \mid \bigoplus_i \Hom_{\mathcal{T}}(M,\Sigma^i N) \text{ finite dimensional for all } M \in \mathcal{T} \} 
\end{align*}
\end{defn}

We can identify these for fd DGAs.

\begin{prop} \label{lhfandrhf}
Let $A$ be a fd DGA such that $A/J$ is $k$-separable. Then
\[
\mathcal{D}_{\f}(A)^{\lhf} = \mathcal{D}_{\cf}(A)^{\lhf} = \mathcal{D}^{\perf}(A) \quad 
\mathcal{D}_{\f}(A)^{\rhf} = \mathcal{D}_{\cf}(A)^{\rhf} = \thick_A(A^\vee)
\]
\end{prop}

\begin{proof} 
For $\mathcal{T} = \mathcal{D}_{\f}(A)$ or $\mathcal{D}_{\cf}(A)$, $\mathcal{T}^{\lhf}$ is a thick subcategory containing $A$ so $\mathcal{D}^{\perf}(A) \subseteq \mathcal{T}^{\lhf}$. Conversely, if $\RHom_A(M,N) \in \mathcal{D}^b(k)$ for all $N \in \mathcal{T}$. Then $\RHom_A(M,A/J_-) \in \mathcal{D}^b(k)$ so by Theorem \ref{descperfdual2}, $M \in \mathcal{D}^{\perf}(A)$. The second result then follows from the equivalence $(-)^\vee$.
\end{proof}

\begin{remark}
This result is stated in 6.9 ii) \cite{KS22} for proper connective DGAs but their methods (relating $\mathcal{D}_{\cf}(A)^{\lhf}$ to $\mathcal{D}_{\cf}(\Aus(A))^{\lhf}$) work for all fd DGAs and provides a different proof of Proposition \ref{lhfandrhf}. The proof of Proposition 6.9 iii) in loc.\ cit.\ does not generalise to prove Theorem \ref{corepresent} as it requires smoothness of $\mathcal{D}_{\f}(A)$. This is not always the case as in Example \ref{powerseriesex}.
\end{remark}

\begin{remark}
Since Definition \ref{defnlhfrhf} is invariant under equivalence we deduce that if $\mathcal{D}_{\f}(A) \simeq \mathcal{D}_{\f}(B)$ then $\mathcal{D}(A) \simeq \mathcal{D}(B)$. 
\end{remark}

We have seen that any fd DGA has a semisimple augmentation $A \to A/J_-$ and $A/J_-$ generates $\mathcal{D}_{\f}(A)$. This sets up a version of Koszul duality for fd DGAs.

\begin{defn}
The Koszul dual of a fd DGA $A$ is the DGA 
\[
A^! := \RHom_A(A/J_-,A/J_-).
\]
\end{defn}

We will be more interested in $(A^!)^{op}$ since it is the contravariant functor $\RHom(-,A/J_-)$ which reflects perfection.

\begin{lemma} \label{Koszuldual1}
Let $A$ be a fd DGA, then the following diagram of equivalences commutes up to quasi-isomorphism.

\[
\begin{tikzcd}[row sep=normal, column sep = huge]
\mathcal{D}_{\f}(A) \arrow{r}{\RHom_A(A/J_-,-)}[swap]{\sim} \arrow{d}{(-)^\vee}[swap]{\wr} & \mathcal{D}^{\perf}(A^!) \arrow{d}{\RHom_{A^!}(-,A^!)}[swap]{\wr} \\
\mathcal{D}_{\f}(A^{op})^{op} \arrow{r}{\sim}[swap]{\RHom_A((-)^\vee,A/J_-)} & \mathcal{D}^{\perf}((A^!)^{op})^{op}
\end{tikzcd}
\]
\end{lemma}

\begin{proof} The top equivalence follows from Theorem \ref{radfilt}. The right vertical functor is an equivalence by Proposition \ref{otherduality}. For the bottom equivalence recall that $A/J_-^\vee$ generates $\mathcal{D}_{\f}(A^{op})$ as in the proof of Theorem \ref{descperfdual2}. Also $(-)^\vee$ induces an quasi-isomorphism of DGAs
\[
\RHom_{A^{op}}(A/J_-^{\vee},A/J_-^{\vee}) \simeq \RHom_A(A/J_-,A/J_-)^{op} = (A^!)^{op}
\]
Then since $\RHom_{A^{op}}((A/J_-)^\vee, -) \simeq \RHom_{A}((-)^\vee, A/J_-)$, the bottom functor is an equivalence. To see that the diagram commutes note that the composite of the bottom and left vertical functors is $\RHom_A(-,A/J_-)$. It induces a quasi-isomorphism
\[
\RHom_A(A/J_-,-) \longrightarrow \RHom_{(A^!)^{op}}(\RHom_A(-,A/J_-),A^!)
\]
This says exactly that the top map is equivalent to composing the left, bottom and right maps. Therefore the diagram commutes.
\end{proof}

By restricting this equivalence we can describe $\mathcal{D}^{\perf}(A)$ in terms of the Koszul dual. For this purpose we introduce the following notation.

\begin{defn}
If $A$ is a DGA let $\mathcal{D}^{\perf}_{\cf}(A) := \mathcal{D}^{\perf}(A) \cap \mathcal{D}_{\cf}(A)$.
\end{defn}

\begin{remark} \label{kronecker} \
\begin{enumerate} 

\item If $K$ is a proper DGA then $\mathcal{D}^{\perf}_{\cf}(K) = \mathcal{D}^{\perf}(K)$ and if $K$ is smooth then $\mathcal{D}^{\perf}_{\cf}(K) = \mathcal{D}_{\cf}(K)$. Indeed this follows from Lemma 3.8 in \cite{KS22}. 

\item There are regular DGAs for which $\mathcal{D}^{\perf}_{\cf}(K)$ is properly contained in both $\mathcal{D}^{\perf}(K)$ and $\mathcal{D}_{\cf}(K)$. As noted in \cite{Orl23}, the infinite Kronecker algebra is an example. 

\item In the language of \cite{KS22} a DGA $K$ is HFD-closed if $\mathcal{D}_{\cf}(K) = \mathcal{D}_{\cf}^{\perf}(K)$ and $\mathcal{D}_{\cf}(K^{op}) = \mathcal{D}_{\cf}^{\perf}(K^{op})$.
\end{enumerate}
\end{remark}

\begin{prop} \label{connective}
Suppose $A$ is a fd DGA concentrated in non-positive degrees. Then $A^!$ is smooth and $\mathcal{D}^{\perf}_{\cf}(A^!) = \mathcal{D}_{\cf}(A^!)$.
\end{prop}

\begin{proof}
For connective DGAs $\mathcal{D}_{\f}(A) = \mathcal{D}_{\cf}(A)$ by Lemma \ref{sf=cf}. Also $\mathcal{D}_{\cf}(A)$ is smooth by Proposition 6.9 in \cite{KS22}. So $\mathcal{D}^{\perf}(A^!)$ is smooth and so $A^!$ is too. By Remark \ref{kronecker}  $\mathcal{D}_{\cf}(A^!) \subseteq \mathcal{D}^{\perf}(A^!)$. 
\end{proof}

\begin{theorem} \label{Koszuldual2}
Let $A$ be a fd DGA such that $A/J$ is $k$-separable. Then there are equivalences
\[
\begin{tikzcd}[column sep = huge]
\mathcal{D}_{\f}(A)^{op} \arrow{r}{\RHom_A(-,A/J_-)}[swap]{\sim} & \mathcal{D}^{\perf}((A^!)^{op})\\
\mathcal{D}^{\perf}(A)^{op} \arrow[u,hook] \arrow[r,"\sim"] & \mathcal{D}^{\perf}_{\cf}((A^!)^{op})  \arrow[u,hook]
\end{tikzcd}
\]
\end{theorem}

\begin{proof}
The top functor is the diagonal equivalence of  Lemma \ref{Koszuldual1}. Theorem 6.1 states that $X \in \mathcal{D}^{\perf}((A^!)^{op})$ is in the image of $\mathcal{D}^{\perf}(A)^{op}$ under this functor exactly when $X$ has finite dimensional cohomology.
\end{proof}

\begin{remark}
The Koszul dual $A^!$ has an augmentation from which we can recover $A$. Applying $\RHom_A(-,A/J_-)$ to $A \to A/J_-$ gives a map of left $A^!$-modules $ A/J_- \leftarrow A^!$. And $\RHom_A(-,A/J_-)$ induces an equivalence of DGAs
\[
A = \RHom_A(A,A) \xrightarrow{\sim} \RHom_{(A^!)^{op}}(A/J_-,A/J_-)^{op}.
\]
\end{remark} 

We now deduce covariant versions of the two equivalences in Theorem \ref{Koszuldual2}. In the covariant version, the square corresponding to that of Theorem \ref{Koszuldual2} need not commute.

\begin{lemma}
If $A$ is a fd DGA, then $(-)^\vee: \mathcal{D}_{\cf}((A^!)^{op}) \simeq \mathcal{D}_{\cf}(A^!)^{op}$ restricts to $\mathcal{D}_{\cf}^{\perf}((A^!)^{op}) \simeq \mathcal{D}_{\cf}^{\perf}(A^!)^{op}$.
\end{lemma}

\begin{proof}
We need to show that for any $X \in \mathcal{D}_{\cf}^{\perf}((A^!)^{op})$ $X^\vee \in \mathcal{D}^{\perf}(A^!)$. Since $A$ generates $\mathcal{D}^{\perf}(A)$ we have that $\RHom_A(A,A/J_-) \simeq A/J_-$ generates $\mathcal{D}_{\cf}^{\perf}((A^!)^{op})^{op}$. Therefore it is enough to show that $A/J_-^\vee \in \mathcal{D}^{\perf}(A^!)$. Now we have $A/J_-^\vee \simeq \RHom_A(A/J_-,A^\vee)$ and so $A/J_-^\vee \in \mathcal{D}((A^!)^{op})$ lies in the image of $\mathcal{D}_{\f}(A)^{op}$ under the functor $\RHom_A(A/J_-,-)$. This is exactly $\mathcal{D}^{\perf}(A^!)$ by Lemma \ref{Koszuldual1}.
\end{proof}

\begin{remark}
The commuting square in Lemma \ref{Koszuldual1} implies that the equivalence $\RHom_{A^!}(-,A^!): \mathcal{D}^{\perf}(A^!) \simeq \mathcal{D}^{\perf}((A^!)^{op})^{op}$ restricts to $\mathcal{D}_{\cf}^{\perf}(A^!)$ if and only if $A$ is Gorenstein. 
\end{remark}

\begin{cor} \label{almostsmoothdual}
If $A$ is a fd DGA, then $\mathcal{D}_{\cf}^{\perf}((A^!)^{op}) = \mathcal{D}_{\cf}((A^!)^{op})$ if and only if $\mathcal{D}_{\cf}^{\perf}(A^!) = \mathcal{D}_{\cf}(A^!)$.
\end{cor}

\begin{proof}
The result follows from the diagram where the horizontal maps are equivalences
\[
\begin{tikzcd}
\mathcal{D}_{\cf}(A^!) \arrow[r,"(-)^\vee"] & \mathcal{D}_{\cf}((A^!)^{op})^{op} \\
\mathcal{D}^{\perf}_{\cf}(A^!) \arrow[u,hook] \arrow[r,"(-)^\vee"] & \mathcal{D}_{\cf}^{\perf}((A^!)^{op})^{op} \arrow[u,hook] \\
\end{tikzcd} \qedhere
\]
\end{proof}

\begin{theorem} \label{Kosequiv}
Suppose $A$ is a fd DGA such that $A/J$ is $k$-separable. Then there are equivalences
\begin{align*}
\RHom_A(A/J_-,-)&: \mathcal{D}_{\f}(A) \xrightarrow{\sim} \mathcal{D}^{\perf}(A^!) \\
- \otimes^\mathbb{L}_A (A/J_-)^\vee &: \mathcal{D}^{\perf}(A) \xrightarrow{\sim} \mathcal{D}_{\cf}^{\perf}(A^!)
\end{align*}  
\end{theorem}

\begin{proof}
The first equivalence is Lemma \ref{Koszuldual1}. Composing the bottom equivalence from Lemma \ref{Koszuldual2} with $(-)^\vee$ states that 
\[
\RHom_A(-,A/J_-)^\vee: \mathcal{D}^{\perf}(A) \xrightarrow{\sim} \mathcal{D}^{\perf}_{\cf}((A^!)^{op})^{op} \xrightarrow{\sim} \mathcal{D}_{\cf}^{\perf}(A^!)
\]
is an equivalence. Then we note that
\[
\RHom_A(-,A/J_-)^\vee \simeq \RHom_A(-,A/J_-^{\vee \vee})^\vee \simeq - \otimes_A^\mathbb{L} A/J_-^\vee \qedhere
\] 
\end{proof}

We can simplify the statement in the case that $A$ is connective.

\begin{cor} \label{connectivekos}
Suppose $A$ is a proper and connective DGA and $k$ is a perfect field. Then there is a smooth DGA $A^!$ such that
\[
\mathcal{D}_{\cf}(A) \simeq \mathcal{D}^{\perf}(A^!) \quad \text{ and } \quad
\mathcal{D}^{\perf}(A) \simeq \mathcal{D}_{\cf}(A^!)
\]
\end{cor}

\begin{proof}
By Corollary 3.12 in \cite{RS20}, we can assume $A$ is a fd DGA concentrated in non-positive degrees such that $A/J$ is $k$-separable. Then the result follows from Proposition \ref{connective} and Theorem \ref{Kosequiv}.
\end{proof}

\begin{ex} \label{powerseriesex}
Here we give an example of a DGA for which the Koszul dual is not smooth. Let $A = k[x]/x^2$ with $\lvert x \rvert = 1$ with zero differential. Then $A/J = k$ and $\RHom_A(k,k) = k[[t]]$ with $\lvert t \rvert = 0$ which is not smooth. Note that $\mathcal{D}_{\cf}^{\perf}(A^!) = \mathcal{D}_{\cf}(A^!)$ since all finite dimensional $k[[t]]$ modules are of the form $\bigoplus_i k[t]/t^{k_i}$ for some $k_i \in \mathbb{N}$ which are all perfect.
\end{ex}

\section{A Corepresentability Result}

We use the Koszul dual to study corepresentability of DG functors $\mathcal{D}_{\f}(A) \to \mathcal{D}^b(k)$. The derived category of a DG category $\mathcal{A}$ will be denoted $\mathcal{D}(\mathcal{A})$. Let $\mathcal{D}_{\cf}(\mathcal{A}) \subseteq \mathcal{D}(\mathcal{A})$ consist of those $\mathcal{A}$-modules $M$ such that $H^\ast M(a)$ is totally finite dimensional for all $a \in \mathcal{A}$. See \cite{Kel06} for an introduction to DG categories. The following fact is well known; we include a proof for completeness.

\begin{lemma} \label{moritaequiv}
If $\mathcal{A}$ is a DG category, the restriction along the Yoneda embedding is an equivalence $\mathcal{D}_{\cf}(\mathcal{D}^{\perf}(\mathcal{A})) \xrightarrow{\sim} \mathcal{D}_{\cf}(\mathcal{A})$. 
\end{lemma}

\begin{proof}
The Yoneda embedding produces a fully-faithful functor 
\[
\mathcal{D}^{\perf}(\mathcal{A}) \hookrightarrow \mathcal{D}(\mathcal{D}^{\perf}(\mathcal{A}))
\]
and since $\mathcal{D}^{\perf}(\mathcal{A})$ is idempotent complete and triangulated its image is $\mathcal{D}^{\perf}(\mathcal{D}^{\perf}(\mathcal{A}))$. The inverse functor is given by restriction along the Yoneda embedding $\mathcal{A} \hookrightarrow \mathcal{D}^{\perf}(\mathcal{A})$. It follows that the coproduct preserving restriction functor 
\[
\mathcal{D}(\mathcal{D}^{\perf}(\mathcal{A})) \to \mathcal{D}(\mathcal{A})
\]
is an equivalence, since it restricts to an equivalence on compact generators. It clearly restricts to cohomologically finite modules. 
\end{proof}

The contravariant Yoneda embedding $Y$ for $\mathcal{D}_{\f}(A)$ restricts as
\[
\begin{tikzcd}
\mathcal{D}_{\f}(A)^{op} \arrow[r,hook,"Y"] & \mathcal{D}(\mathcal{D}_{\f}(A)^{op}) \\
\mathcal{D}^{\perf}(A) \arrow[u,hook] \arrow[r,hook] & \mathcal{D}_{\cf}(\mathcal{D}_{\f}(A)^{op}) \arrow[u,hook] 
\end{tikzcd}
\]
since if $M \in \mathcal{D}^{\perf}(A)$ then $\RHom(M,N) \in \mathcal{D}^b(k)$ for all $N \in \mathcal{D}_{\f}(A)$. 

\begin{theorem} \label{corepresent}
Let $A$ be a fd DGA such that $A/J$ is $k$-separable. Then
\[ \tag{$\ast$} \label{evaluation}
\mathcal{D}^{\perf}(A)^{op} \hookrightarrow \mathcal{D}_{\cf}(\mathcal{D}_{\f}(A)^{op})
\]
is an equivalence if and only if $\mathcal{D}_{\cf}(A^!) = \mathcal{D}_{\cf}^{\perf}(A^!)$. 
\end{theorem}

\begin{proof}
We have an equivalence 
\[
\mathcal{D}_{\cf}(\mathcal{D}_{\f}(A)^{op}) \simeq \mathcal{D}_{\cf}(\mathcal{D}^{\perf}((A^!)^{op})) \simeq \mathcal{D}_{\cf}((A^!)^{op})
\]
where the first is induced by the equivalence
\[
\RHom(-,A/J_-): \mathcal{D}_{\f}(A)^{op} \xrightarrow{\sim} \mathcal{D}^{\perf}((A^!)^{op})
\]
and the second is Lemma \ref{moritaequiv}. The long composition
\[
\mathcal{D}^{\perf}(A)^{op} \hookrightarrow \mathcal{D}_{\cf}(\mathcal{D}_{\f}(A)^{op}) \simeq \mathcal{D}_{\cf}((A^!)^{op}
\]
is essentially surjective if and only if the first map is. The long composition sends
\begin{align*}
M & \mapsto \RHom_A(M,-) \\ 
 & \simeq \RHom_{(A^!)}(\RHom_A(-,A/J_-),\RHom(M,A/J_-)) \in \mathcal{D}_{\cf}(\mathcal{D}_{\f}(A)^{op})\\
& \mapsto \RHom_{(A^!)^{op}}(-,\RHom_A(M,A/J_-)) \in \mathcal{D}_{\cf}(\mathcal{D}^{\perf}((A^!)^{op})) \\
& \mapsto \RHom_A(M,A/J_-) \in \mathcal{D}_{\cf}((A^!)^{op})
\end{align*}
where the first equivalence follows since $\RHom_A(-,A/J_-)$ is fully faithful on $\mathcal{D}_{\f}(A)$. Therefore, (\ref{evaluation}) is essentially surjective if and only if 
\[
\RHom_A(-,A/J_-): \mathcal{D}^{\perf}(A)^{op} \to \mathcal{D}_{\cf}((A^!)^{op})
\]
is essentially surjective. By Theorem \ref{Koszuldual2}, its image is $\mathcal{D}_{\cf}^{\perf}((A^!)^{op})$. By Proposition \ref{almostsmoothdual}, $\mathcal{D}_{\cf}^{\perf}((A^!)^{op}) = \mathcal{D}_{\cf}((A^!)^{op})$ if and only if $\mathcal{D}_{\cf}^{\perf}(A^!) = \mathcal{D}_{\cf}(A^!)$.
\end{proof}

\begin{remark}
This theorem is related to reflexivity as defined in \cite{KS22} and will be investigated in forthcoming work. In this language, $\mathcal{D}_{\cf}(A^!) = \mathcal{D}^{\perf}_{\cf}(A^!)$ if and only if $A^!$ is HFD-closed (using Proposition \ref{almostsmoothdual}). If $A$ is concentrated in non-positive degrees, the condition on the Koszul dual is satisfied by Proposition \ref{connective}. The author is unaware of a fd DGA $A$ for which  $\mathcal{D}^{\perf}_{\cf}(A^!) \neq \mathcal{D}_{\cf}(A^!)$. 
\end{remark}

\appendix

\section{Separability of Graded Constructions}

We prove the following lemma which is used in Proposition \ref{envelopesemi}.

\begin{lemma} \label{appendixlemma}
If $A$ is a graded algebra whose underlying algebra is $k$-separable then the underlying algebra of the graded enveloping algebra $A^e$ is semisimple. 
\end{lemma}

Lemma \ref{appendixlemma} follows immediately from Proposition \ref{appprop}. The only difference to the usual proof is the existence of signs.

\begin{prop}\label{appprop}
Let $A$ and $B$ be graded $k$-algebras whose underlying algebras are $k$-separable. Then the underlying algebra of the graded opposite algebra is $k$-separable. Also the underlying algebra of the graded tensor product is $k$-separable. 
\end{prop}

\begin{proof}
We will prove the second statement. The first uses a simpler version of the same argument. Let $- \otimes_k -$ denote the ungraded tensor product and $- \otimes^G -$ underlying algebra of the graded tensor product of graded algebras. Let $p^A = \sum a_i \otimes a_i' \in A \otimes_k A$ and $p^B = \sum b_j \otimes b_j' \in B \otimes_k B$ be the separability idempotents of $A$ and $B$.  Since the multiplication is graded and $\mu_A(p^A) = 1$ we have that  $p^A \in \bigoplus A^n \otimes_k A^{-n}$. Hence we can assume that each $a_i$ is homogenous and $\lvert a_i \rvert = - \lvert a_i' \rvert$ and similarly for $b_j$. We claim that
\[
p := \sum_{i,j} (-1)^{\lvert b_j \rvert \lvert a_i'\rvert} a_i \otimes^G b_j \otimes a_i' \otimes^G b_j' \in (A \otimes^G B) \otimes_k (A \otimes^G B)
\]
is a separability idempotent of $A \otimes^G B$. Indeed $\mu_{A \otimes^G B}(p)$ equals
\[
\sum_{i,j} (-1)^{\lvert b_j \rvert \lvert a_i'\rvert + \lvert b_j \rvert \lvert a_i'\rvert} a_ia_i' \otimes^G b_j b_j'   = \left(\sum_i a_ia_i'\right) \otimes^G \left(\sum_j b_jb_j'\right) = 1
\]

Now let $x \otimes^G y \in A \otimes^G B$ such that $x$ and $y$ are homogenous. Then
\[
(x \otimes y)p = \sum_{i,j} (-1)^{\lvert b_j \rvert \lvert a_i'\rvert + \lvert a_i \rvert \lvert y\rvert} xa_i \otimes^G yb_j \otimes a_i' \otimes^G b_j'.
\]
We use the algebra isomorphism $A \otimes^G B \otimes_k A \otimes^G B \simeq B \otimes^G A \otimes_k A \otimes^G B$
\[
a \otimes^G b \otimes a' \otimes^G b' \mapsto (-1)^{\lvert a \rvert \lvert b \rvert} b \otimes^G a \otimes a' \otimes b',
\]
under which $(x \otimes y)p$ maps to
\begin{align*}
& \sum_{i,j} (-1)^{\lvert b_j \rvert \lvert a_i'\rvert + \lvert a_i \rvert \lvert y\rvert + (\lvert x \rvert + \lvert a_i \rvert)(\lvert y \rvert + \lvert b_j \rvert)} yb_j \otimes^G xa_i \otimes a_i' \otimes^G b_j' \\
&= \sum_{j} (-1)^{\lvert x \rvert \lvert y \rvert + \lvert x \rvert \lvert b_j \rvert} yb_j \otimes^G x \left( \sum_i  a_i \otimes a_i' \right) \otimes^G b_j' \\
&= \sum_{j} (-1)^{\lvert x \rvert \lvert y \rvert + \lvert x \rvert \lvert b_j \rvert} yb_j \otimes^G \left( \sum_i  a_i \otimes a_i' \right) x \otimes^G b_j' 
\end{align*}
where we have used the fact that $\lvert a_i \rvert = - \lvert a_i'\rvert$. Now  pass back along the isomorphism to see that $(x \otimes y)p$ equals
\[ \label{eqgrad} \tag{$\ast$}
\sum_{i,j} (-1)^{\lvert x \rvert \lvert y \rvert + \lvert x \rvert \lvert b_j \rvert + \lvert y \rvert \lvert a_i \rvert + \lvert b_j \rvert \lvert a_i \rvert} a_i \otimes^G yb_j \otimes a_i'x \otimes^G b_j'.
\]
We now repeat the same trick with the isomorphism 
\begin{align*}
A \otimes^G B \otimes_k \otimes A \otimes^G B \simeq A \otimes^G B \otimes_k \otimes B \otimes^G A; \\
a \otimes^G b \otimes a' \otimes^G b' \mapsto (-1)^{\lvert a' \rvert \lvert b' \rvert} a \otimes^G b \otimes b' \otimes^G a'
\end{align*}

\begin{comment}
under which we see that (\ref{eqgrad}) is sent to
\begin{align*}
& \sum_{i,j} (-1)^{\lvert x \rvert \lvert y \rvert + \lvert x \rvert \lvert b_j \rvert + \lvert y \rvert \lvert a_i \rvert + \lvert b_j \rvert \lvert a_i \rvert + \lvert b_j' \rvert \lvert a_i' \rvert + \lvert b_j' \rvert \lvert x \rvert } a_i \otimes^G yb_j \otimes b_j' \otimes^G a_i'x \\
& = \sum_i (-1)^{\lvert x \rvert \lvert y \rvert + \lvert y \rvert \lvert a_i \rvert} a_i \otimes^G y \left(\sum_j b_j \otimes b_j' \right) \otimes^G a_i'x \\
& = \sum_i (-1)^{\lvert x \rvert \lvert y \rvert + \lvert y \rvert \lvert a_i \rvert} a_i \otimes^G  \left(\sum_j b_j \otimes b_j' \right)y \otimes^G a_i'x \\
& = \sum_{i,j} (-1)^{\lvert x \rvert \lvert y \rvert + \lvert y \rvert \lvert a_i \rvert} a_i \otimes^G b_j \otimes b_j'y \otimes^G a_i'x.
\end{align*}
Passing back along this isomorphism we see that $(x \otimes^G y)p$ is equal to 
\begin{align*}
& \sum_{i,j} (-1)^{\lvert x \rvert \lvert y \rvert + \lvert y \rvert \lvert a_i \rvert + (\lvert b_j' \rvert + \lvert y \rvert)(\lvert a_i' \rvert + \lvert x \rvert)} a_i \otimes^G b_j \otimes a_i' x \otimes^G b_j'y  \\
& = \sum_{i,j} (-1)^{ \lvert b_j' \rvert \lvert a_i' \rvert + \lvert b_j' \rvert \lvert x \rvert } a_i \otimes^G b_j \otimes a_i' x \otimes^G b_j'y  \\
& = \sum_{i,j} (-1)^{ \lvert b_j' \rvert \lvert a_i \rvert} a_i \otimes^G b_j \otimes (a_i' \otimes^G b_j')(x \otimes^G y) = p (x \otimes^G y).
\end{align*}
Hence $p$ is a separability idempotent of $A \otimes^G B$. 
\end{comment}
To deduce that (\ref{eqgrad}) is equal to 
\[
\sum_{i,j} (-1)^{ \lvert b_j' \rvert \lvert a_i \rvert} a_i \otimes^G b_j \otimes (a_i' \otimes^G b_j')(x \otimes^G y) = p (x \otimes^G y). \qedhere
\]
\end{proof}

\bibliographystyle{alpha}
\bibliography{biblio}

\begin{thebibliography}{Mar08}

\bibitem[BB02]{VB02}
Alexey Bondal and Michel Bergh.
\newblock Generators and representability of functors in commutative and noncommutative geometry.
\newblock {\em Moscow Mathematical Journal}, 3:1:36, 05 2002.

\bibitem[BK89]{BoKa}
Alexey Bondal and M~Kapranov.
\newblock Representable functors, serre functors, and mutations.
\newblock {\em Mathematics of the USSR-Izvestiya}, 35:519, 10 1989.

\bibitem[Efi20]{Efi18}
Alexander Efimov.
\newblock Categorical smooth compactifications and generalized hodge-to-de rham degeneration.
\newblock {\em Invent. math. 222, 667–694 (2020).}, 3, 05 2020.

\bibitem[Jin20]{Jin20}
Haibo Jin.
\newblock Cohen-macaulay differential graded modules and negative calabi-yau configurations.
\newblock {\em Advances in Mathematics}, 374:107338, 2020.

\bibitem[Kel06]{Kel06}
Bernhard Keller.
\newblock On differential graded categories.
\newblock {\em arXiv: 0601185}, 2006.

\bibitem[KS22]{KS22}
Alexander Kuznetsov and Evgeny Shinder.
\newblock Homologically finite-dimensional objects in triangulated categories.
\newblock {\em arXiv: 2211.09418}, 2022.

\bibitem[Kuz09]{Kuz09}
Alexander Kuznetsov.
\newblock Hochschild homology and semiorthogonal decompositions.
\newblock {\em arXiv: 0904.4330}, 2009.

\bibitem[Mar01]{Ma01}
N~Markarian.
\newblock {Poincaré-Birkoff-Witt isomorphism, Hochshild homology and Riemann-Roch theorem}.
\newblock {\em preprint}, 2001.

\bibitem[Mar08]{Ma08}
Nikita Markarian.
\newblock The atiyah class, hochschild cohomology and the riemann-roch theorem.
\newblock {\em Journal of the London Mathematical Society}, 79(1):129–143, October 2008.

\bibitem[Orl20]{Orl20}
Dmitri Orlov.
\newblock Finite-dimensional differential graded algebras and their geometric realizations.
\newblock {\em Advances in Mathematics}, 366:107096, jun 2020.

\bibitem[Orl23]{Orl23}
Dmitri Orlov.
\newblock Smooth dg algebras and twisted tensor product.
\newblock {\em arXiv: 2305.19799}, 2023.

\bibitem[RS20]{RS20}
Theo Raedschelders and Greg Stevenson.
\newblock Proper connective differential graded algebras and their geometric realizations.
\newblock {\em arXiv: 1903.02849}, 2020.

\bibitem[Shk07]{shk07}
D.~Shklyarov.
\newblock On serre duality for compact homologically smooth dg algebras.
\newblock {\em arXiv: 0702590}, 2007.

\end{thebibliography}

\end{document}